\documentclass[12pt]{article}

\usepackage{amsmath,amssymb,amsthm}

\usepackage{verbatim}

\usepackage{mathtools}

\usepackage{tensor}

\newcommand{\R}{\mathbb{R}}

\newcommand{\Hy}{\mathbb{H}}

\newcommand{\lp}{\left (}
\newcommand{\rp}{\right )}

\def\Xint#1{\mathchoice
{\XXint\displaystyle\textstyle{#1}}%
{\XXint\textstyle\scriptstyle{#1}}%
{\XXint\scriptstyle\scriptscriptstyle{#1}}%
{\XXint\scriptscriptstyle\scriptscriptstyle{#1}}%
\!\int}
\def\XXint#1#2#3{{\setbox0=\hbox{$#1{#2#3}{\int}$ }
\vcenter{\hbox{$#2#3$ }}\kern-.6\wd0}}

\def\dashint{\Xint-}

\author{Brian Allen}

\date{Spring 2019}

\title{Sobolev stability of the PMT and RPI using IMCF}

\newtheorem{Thm}{Theorem}[section]
\newtheorem{Cor}[Thm]{Corollary}

\newtheorem{Prop}[Thm]{Proposition}
\newtheorem{rmrk}[Thm]{Remark}
\newtheorem{Lem}[Thm]{Lemma}

\newtheorem{Def}[Thm]{Definition}

\begin{document}

\maketitle

\begin{abstract}
We study the Sobolev stability of the Positive Mass Theorem (PMT) and the Riemannian Penrose Inequality (RPI) in the case where a region of a sequence of manifolds $M^3_i$ can be foliated by a smooth solution of Inverse Mean Curvature Flow (IMCF) which is uniformly controlled for time $t \in [0,T]$. 

In particular, we consider a sequence of regions of manifolds $U_T^i\subset M_i^3$, foliated by a IMCF, $\Sigma_t$, such that if $\partial U_T^i = \Sigma_0^i \cup \Sigma_T^i$ and $m_H(\Sigma_T^i) \rightarrow 0$ then $U_T^i$ converges in $W^{1,2}$ to a flat annulus or in the hyperbolic setting it converges to a annulus portion of hyperbolic space. If instead $m_H(\Sigma_T^i)-m_H(\Sigma_0^i) \rightarrow 0$ and $m_H(\Sigma_T^i) \rightarrow m >0$ then we show that $U_T^i$ converges in $W^{1,2}$ to a topological annulus portion of the Schwarzschild metric  or in the Hyperbolic case to a topological annulus portion of the Anti-de~Sitter Schwarzschild metric.
\end{abstract}

\section{Introduction}\label{sec:intro}
If we consider a complete, asymptotically flat manifold with nonnegative scalar curvature then the Positive Mass Theorem (PMT) says that $M^3$ has positive ADM mass. This was proved by Schoen-Yau \cite{SY} using minimal surface techniques. The rigidity statement says that if $m_{ADM}(M) = 0$ then $M$ is isometric to Euclidean space. Similarly, the Riemannian Penrose Inequality (RPI) says that if $\partial M$ consists of an  outermost minimal surface $\Sigma_0$ then
\begin{align}
 m_{ADM}(M) \ge \sqrt{\frac{|\Sigma_0|}{16 \pi}} \label{RPI Intro}
\end{align}
where $|\Sigma_0|$ is the area of $\Sigma_0$. In the case of equality, i.e. $m_{ADM}(M) = \sqrt{\frac{|\Sigma_0|}{16 \pi}}$, then $M$ is isometric to the Schwarzschild metric \eqref{SchwarzMetricDef}. This was first proved by Geroch \cite{G} in the rotationally symmetric case using Inverse Mean Curvature Flow (IMCF) and the Geroch monotonicity of the Hawking mass
\begin{align}
m_H(\Sigma) = \sqrt{\frac{|\Sigma|}{(16\pi)^3}} \left (16 \pi - \int_{\Sigma} H^2 d \mu \right ).\label{HawkingMass Intro}
\end{align}
Huisken-Ilmanen \cite{HI} then extended these ideas to general asymptotically flat manifolds with a connected horizon using novel weak solutions to IMCF. Soon after Bray \cite{B} proved the general case of the RPI using a conformal flow method. 

In the asymptotically hyperbolic case the notion of mass was defined mathematically and explored by Chru\'{s}ciel-Herszlich \cite{CH} and Wang \cite{W}. Earlier explorations of mass in this context were carried out by Abbott-Deser \cite{AD}, Ashtekar-Magnon \cite{AM}, and Gibbons-Hawking-Horowitz-Perry \cite{GHHP}.  The PMT in this context for manifolds with scalar curvature greater than or equal to $-6$ has been proved by Wang \cite{W}, Chru\'{s}ciel-Herszlich \cite{CH}, Andersson-Cai-Galloway \cite{ACG}, and Sakovich  in various different cases. The notion of Hawking mass in this context is defined as
\begin{align}
m_H^{\Hy}(\Sigma) = \sqrt{\frac{|\Sigma|}{(16\pi)^3}} \left (16 \pi - \int_{\Sigma} H^2-4 d \mu \right ).\label{HypHawkingMass Intro}
\end{align}
 The RPI conjecture in the case of asymptotically hyperbolic manifolds satisfying the scalar curvature bound says that the appropriate mass for this context satisfies \eqref{RPI Intro}. In the case of equality the manifold is isometric to the Anti-de~Sitter Scharzschild metric \eqref{AntiMetricDef}. Neves \cite{N} observed that the method of using IMCF to prove the RPI in the asymptotically hyperbolic case is not sufficient. Later, Hung and Wang \cite{HW} discuss this issue further in a note on IMCF in hyperbolic space. This conjecture is still open but special cases and related estimates have been obtained by Dahl-Gicquad-Sakovich \cite{DGS1}, de~Lima-Gir\~{a}o \cite{dLG}, and Brendle-Hung-Wang \cite{BHW}.

In this paper we are concerned with the stability of these four rigidity statements. Lee and Sormani \cite{LS1} show that one cannot obtain smooth stability of the PMT even in the asymptotically flat, spherically symmetric setting. In that setting they prove Sormani-Wenger intrinsic flat (SWIF) convergence stability using Geroch monotonicity. LeFloch and Sormani \cite{LeS} prove Sobolev stability using Geroch monotonicity but only in the asymptotically flat, sphereically symmetric setting. Additional related work will be mentioned below. The main goal of this paper is to improve upon the author's previous results on $L^2$ stability \cite{BA2,BA3} in order to show $W^{1,2}$ stability. 

In \cite{HI}, Huisken-Ilmanen show how to use weak solutions of IMCF in order to prove the PMT for asymptotically flat Riemanian manifolds as well as the RPI in the case of a connected boundary. The weak solutions defined  by Huisken-Ilmanen jump over gravity wells and hence do not produce a complete foliation of the ambient manifold. This is not a problem for Huisken-Ilmanen since they are able to show that the Geroch monotonicity of the Hawking mass is preserved through these jumps. For our purposes, we need the IMCF to foliate the ambient manifold and hence we focus on regions of manifolds which can be foliated by smooth solutions of IMCF which are uniformly controlled.  For a glimpse of long time existence and asymptotic analysis results for smooth IMCF in various ambient manifolds see the work of the author \cite{BA4'}, Ding \cite{QD}, Gerhardt \cite{CG1, CG2},  Scheuer \cite{S1,S2}, Urbas \cite{U}, and Zhou \cite{Z}.

\begin{Def} \label{IMCFClass} If we have $\Sigma^2$ a surface in a Riemannian manifold, $M^3$, we will denote the induced metric, mean curvature, second fundamental form, principal curvatures, Gauss curvature, area, Hawking mass and Neummann isoperimetric constant as $g$, $H$, $A$, $\lambda_i$, $K$, $|\Sigma|$, $m_H(\Sigma)$, $IN_1(\Sigma)$, respectively. We will denote the Riemann curvature, Ricci curvature, scalar curvature, sectional curvature tangent to $\Sigma$, and ADM mass as $Rm$, $Rc$, $R$, $K_{12}$, $m_{ADM}(M)$, respectively.

Define the two classes of manifolds with boundary foliated by IMCF as follows

\begin{align*}
\mathcal{M}_{\Hy,r_0,H_0,I_0}^{T,H_1,A_1}:=\{& M \text{ a Riemannian manifold, } U_T \subset M, R \ge -6|
\\ &\exists \Sigma \subset M \text{compact, connected surface, } 
\\& IN_1(\Sigma) \ge I_0, m_H^{\Hy}(\Sigma) \ge 0 \text{,and } |\Sigma|=4\pi r_0^2. 
\\ &\exists \Sigma_t \text{ smooth solution to IMCF, such that }\Sigma_0=\Sigma,
\\& H_0 \le H(x,t) \le H_1 < \infty, \|A\|_{W^{2,2}(\Sigma\times [0,T])} \le A_1,
\\&\text{and } U_T = \{x\in\Sigma_t: t \in [0,T]\} \}
\end{align*}
and
\begin{align*}
\mathcal{M}_{r_0,H_0,I_0}^{T,H_1,A_1}:=\{& U_T \in \mathcal{M}_{\Hy,r_0,H_0,I_0}^{T,H_1,A_1}| R \ge 0, m_H(\Sigma) \ge 0 \}
\end{align*}
where $0 < H_0 < H_1 < \infty$, $0 < I_0,A_1,A_2, r_0 < \infty$ and $0 < T < \infty$.
\end{Def}

\begin{rmrk}
All norms in this paper are defined on $\Sigma \times [0,T]$ with respect to the Euclidean metric $\delta$ which is given in IMCF coordinates below. The diffeomorphism we use to impose coordinates on $U_T$ is discussed in subsection \ref{subsec: Old Estimates} before Proposition \ref{avgH}.
\end{rmrk}

\begin{rmrk}
The reader should make note that the difference between the class of IMCF's considered in this paper, as opposed to the author's previous paper on $L^2$ stability \cite{BA2, BA3}, is the addition of a $W^{2,2}$ bound on $A$ which is asking for some uniform higher regularity of the family of IMCF's. Notice by Morrey's inequality this implies a $C^{0,\alpha}$ bound on $|A|$ as in the previous paper \cite{BA2,BA3}.
\end{rmrk}

Observe that the Riemannian metric, $\hat{g}_i$, on these manifolds can now be expressed using a gauge defined on $\Sigma \times [0,T]$ by IMCF as,
\begin{align}
 \hat{g}^i&=\frac{1}{H(x,t)^2}dt^2 + g^i(x,t),\label{AmbientMetricDef}
\end{align}
where $g^i(x,t)$ is the metric on $\Sigma_t^i$.
On Euclidean space, concentric spheres flow according to IMCF and so the Euclidean metric $\delta$ can be expressed using a gauge defined on $\Sigma \times [0,T]$ by IMCF as,
\begin{align}
\delta &= \frac{r_0^2}{4}dt^2 + r_0^2e^t \sigma, \label{deltaMetricDef}
\end{align}
where $\sigma$ is the round metric on $\Sigma$.

Similarly, one can express the Schwarzschild, hyperbolic, and Anti-de~Sitter Schwarzschild metrics
\begin{align}
g_S &= \frac{r_0^2}{4}\left (1 - \frac{2}{r_0} m e^{-t/2} \right )^{-1} e^tdt^2 + r_0^2e^t \sigma \label{SchwarzMetricDef}
\\g_{\Hy} &= \frac{1}{4}\left(1+ \frac{e^{-t}}{r_0^2} \right)^{-1}dt^2 + r_0^2e^t \sigma,\label{HypMetricDef}
\\ g_{AdSS} &= \frac{1}{4}\left (1+\frac{e^{-t}}{r_0^2} - \frac{2}{r_0^3} m e^{-3t/2} \right )^{-1} dt^2 + r_0^2e^t \sigma.\label{AntiMetricDef}
\end{align}

In our first theorem we prove stability of the PMT: that when the Hawking mass of the outer boundary converges to $0$ the regions converge to annular regions in Euclidean space. Prior work in this direction under a variety of hypotheses was conducted by Bray-Finster \cite{BF}, Finster-Kath \cite{FK}, Corvino \cite{C}, Finster \cite{F}, Lee \cite{L}, Lee-Sormani \cite{LS1}, Huang-Lee-Sormani \cite{HLS}, the author \cite{BA2,BA4}, and Bryden \cite{B}. 

\begin{Thm}\label{SPMTEuc}
Let $U_{T}^i \subset M_i^3$ be a sequence such that $U_{T}^i\subset \mathcal{M}_{r_0,H_0,I_0}^{T,H_1,A_1}$ and 
\begin{align}
m_H(\Sigma_{T}^i) \rightarrow 0 \text{ as }i \rightarrow \infty. 
\end{align} 
If we assume,
\begin{align}
&\| Rc^i(\nu,\nu)\|_{W^{1,2}(\Sigma\times [0,T])} \le C,\label{FundBound1}
\\&\| R^i\|_{L^2(\Sigma\times [0,T])} \le C,\label{FundBound2}
\\&diam(\Sigma_t^i) \le D \text{ } \forall \text{ } i, t \in [0,T],\label{FundBound3}
\\&  |K^i| \le C \text{ on } \Sigma_T,\label{FundBound4}
\end{align}
where $W^{1,2}(\Sigma\times [0,T])$ is defined with respect to $\delta$, then
\begin{align}
\hat{g}^i \rightarrow \delta
\end{align}
in $W^{1,2}$ with respect to $\delta$ and thus volumes converge.  
\end{Thm}

In our second theorem we prove stability of the RPI: that when the Hawking mass of the outer  and inner boundary converge to the same value $m$ the regions converge to annular regions in the Schwarzschild manifold. Prior work in this direction has been done by Lee-Sormani \cite{LS2} in the rotationally symmetric case and by the author \cite{BA2} using IMCF.

\begin{Thm}\label{SRPIEuc}
Let $U_{T}^i \subset M_i^3$ be a sequence such that $U_{T}^i\subset \mathcal{M}_{r_0,H_0,I_0}^{T,H_1,A_1}$, 
\begin{align}
m_H(\Sigma_{T}^i)- m_H(\Sigma_0^i) \rightarrow 0\text{, and }m_H(\Sigma_0) \rightarrow m > 0 \text{ as }i \rightarrow \infty. 
\end{align}
If we assume \eqref{FundBound1}, \eqref{FundBound2}, \eqref{FundBound3}, and \eqref{FundBound4} then
\begin{align}
\hat{g}^i \rightarrow g_S
\end{align}
in $W^{1,2}$ with respect to $\delta$ and thus volumes converge. 

\end{Thm}

In our third theorem we prove stability of the PMT in the asymptotically hyperbolic case: that when the Hawking mass of the outer boundary converges to $0$ the regions converge to annular regions in the hyperbolic space. Prior work in this direction was conducted by Dahl-Gicquad-Sakovich \cite{DGS}, Sakovich-Sormani \cite{SS}, the author \cite{BA3}, and Cabrera Pacheco \cite{AJCP}.

\begin{Thm}\label{SPMTHyp}
Let $U_{T}^i \subset M_i^3$ be a sequence such that $U_{T}^i\subset \mathcal{M}_{\Hy, r_0,H_0,I_0}^{T,H_1,A_1}$ and 
\begin{align}
m_H^{\Hy}(\Sigma_{T}^i) \rightarrow 0 \text{ as }i \rightarrow \infty.
\end{align}  
If we assume \eqref{FundBound1}, \eqref{FundBound2}, \eqref{FundBound3}, and \eqref{FundBound4} then
\begin{align}
\hat{g}^i \rightarrow g_{\Hy}
\end{align}
in $W^{1,2}$ with respect to $\delta$ and thus volumes converge. 

\end{Thm}

In our fourth theorem we prove stability of the RPI in the asymptotically hyperbolic case: that when the Hawking mass of the outer  and inner boundary converge to the same value $m$ the regions converge to annular regions in the Anti-de~Sitter Schwarzschild manifold. Prior work in this direction has been conducted by the author \cite{BA3} using IMCF.

\begin{Thm}\label{SRPIHyp}
Let $U_{T}^i \subset M_i^3$ be a sequence such that $U_{T}^i\subset \mathcal{M}_{\Hy, r_0,H_0,I_0}^{T,H_1,A_1}$, 
\begin{align}
m_H^{\Hy}(\Sigma_{T}^i)- m_H^{\Hy}(\Sigma_0^i) \rightarrow 0\text{ and }m_H^{\Hy}(\Sigma_0) \rightarrow m > 0 \text{ as }i \rightarrow \infty. 
\end{align}
If we assume \eqref{FundBound1}, \eqref{FundBound2}, \eqref{FundBound3}, and \eqref{FundBound4} then
\begin{align}
\hat{g}^i \rightarrow g_{ADSS}
\end{align}
in $W^{1,2}$ with respect to $\delta$ and thus volumes converge. 
\end{Thm}

\begin{rmrk}
One should not expect $W^{1,2}$ convergence to imply SWIF convergence since the author and Sormani \cite{BS} have shown that $L^2$ convergence does not agree with GH and/or SWIF convergence (see example 3.4 in \cite{BS}) since valleys can form on sets of measure zero. By a similar example, one can see that $W^{1,2}$ convergence in dimension three will not imply SWIF convergence either. By what the main theorem  of the author and Sormani \cite{BS}, one expects to need to combine $L^p$ convergence with $C^0$ convergence from below in order to be able to conclude SWIF convergence which is what the author carries out in \cite{BA4} for the PMT under various assumptions.

Since $W^{1,2}$ convergence provides additional convergence information about the geometry of the sequence, as shown by LeFloch-Mardare \cite{LeM}, it is useful to show both SWIF and $W^{1,2}$ convergence when appropriate.  Also, notice that the curvature hypotheses of the main theorems will clearly hold in the ends of asymptotically flat or asymptotically hyperbolic manifolds whose asymptotic decay rates are uniformly controlled.
 \end{rmrk}

In Section 2 we will use IMCF to get important higher order estimates of the metric $\hat{g}^i$ on the foliated region $U_T^i\subset M_i$ which build upon the estimates of the previous papers \cite{BA2, BA3}. We also review some key estimates obtained in \cite{BA2, BA3} that are needed in this paper. 

In Section 3 we use the estimates of the previous section to show convergence of $\hat{g}$ to the appropriate prototype space $\delta, g_S, g_{\Hy},$ or $g_{ADSS}$. This is done by showing convergence of $\hat{g}$ to simpler metrics, successively, until we get to $\delta, g_S, g_{\Hy},$ or $g_{ADSS}$, and combining this chain of estimates by the triangle inequality.

\section{Higher Order Estimates for Manifolds Foliated by IMCF}\label{sect-Estimates}

In this section we expand upon the estimates found in the previous papers of the author on $L^2$ convergence \cite{BA2, BA3}. For the readers convenience we will repeat some of the estimates obtained in \cite{BA2} since we will also need them here but a majority of the estimates will be new. 

 We remember that IMCF is defined for surfaces $\Sigma^n \subset M^{n+1}$ evolving through a one parameter family of embeddings $F: \Sigma \times [0,T] \rightarrow M$, $F$ satisfying inverse mean curvature flow

\begin{equation}
\begin{cases}
\frac{\partial F}{\partial t}(p,t) = \frac{\nu(p,t)}{H(p,t)}  &\text{ for } (p,t) \in \Sigma \times [0,T)
\\ F(p,0) = \Sigma_0  &\text{ for } p \in \Sigma 
\end{cases}
\end{equation}
where $H$ is the mean curvature of $\Sigma_t := F_t(\Sigma)$ and $\nu$ is the outward pointing normal vector. The outward pointing normal vector will be well defined in our case since we will be considering regions of asymptotically flat  or asymptotically hyperbolic manifolds with one end. 

We also remind the reader of the definition of the Hawking mass in the asymptotically Euclidean setting,
\begin{align}
m_H(\Sigma) = \sqrt{\frac{|\Sigma|}{(16\pi)^3}} \left (16 \pi - \int_{\Sigma} H^2 d \mu \right ),
\end{align}
as well as the Hawking mass in the asymptotically hyperbolic setting,
\begin{align}
m_H^{\Hy}(\Sigma) = \sqrt{\frac{|\Sigma|}{(16\pi)^3}} \left (16 \pi - \int_{\Sigma} H^2-4 d \mu \right ).
\end{align}

\subsection{Previous Estimates for $L^2$ Convergence}\label{subsec: Old Estimates}

As a notational convenience we will use $\mathcal{H}^2 =H^2, H^2 - 4$ depending on whether we are considering the Euclidean or hyperbolic setting when it is clear from the context which model we have in mind. All of the results in this subsection are from the author's previous papers on $L^2$ stability \cite{BA2, BA3} so the reader is directed to \cite{BA2, BA3} for proofs of the following results.

We begin by noting some simple consequences of the assumptions on the Hawking mass.

\begin{Lem} \label{naiveEstimate}
Let $\Sigma^2 \subset M^3$ be a hypersurface and $\Sigma_t$ it's corresponding solution of IMCF. If 
\begin{align}
m_1 \le &m_H(\Sigma_t) \le m_2,
\\0 < H_0 \le &H(x,t) \le H_1 < \infty
\end{align}
 then 
\begin{align}
|\Sigma_t|&=|\Sigma_0| e^t
\\16 \pi \left (1 - \sqrt{\frac{16 \pi}{|\Sigma_0|}}m_2e^{-t/2}  \right ) &\le \int_{\Sigma_t} H^2 d \mu \le 16 \pi \left (1 - \sqrt{\frac{16 \pi}{|\Sigma_0|}}m_1e^{-t/2}  \right ) \label{Eq-NaiveAvg}
\\ \frac{16 \pi}{|\Sigma_0|} \left (1 - \sqrt{\frac{16 \pi}{|\Sigma_0|}}m_2e^{-t/2}  \right )e^{-t} &\le \dashint_{\Sigma_t} H^2 d \mu \le  \frac{16 \pi}{|\Sigma_0|} \left (1 - \sqrt{\frac{16 \pi}{|\Sigma_0|}}m_1e^{-t/2}  \right )e^{-t}\label{AvgHEst}
\end{align}
where $|\Sigma_t|$ is the $n$-dimensional area of $\Sigma$.

Hence if 
\begin{align}
m_{H}(\Sigma_T) \rightarrow 0
\end{align} then 
\begin{align}
\bar{H^2}_i(t):=\dashint_{\Sigma_t^i} \mathcal{H}_i^2 d \mu \rightarrow  \frac{4}{r_0^2}e^{-t}\label{unifAvgHEst1}
\end{align}
for every $t\in [0,T]$.

If 
\begin{align}
&m_{H}(\Sigma_T)-m_{H}(\Sigma_0) \rightarrow 0\text{ and }m_{H}(\Sigma_0) \rightarrow m > 0
\end{align}
 then  
 \begin{align}
\bar{H^2}_i(t):=\dashint_{\Sigma_t^i} \mathcal{H}_i^2 d \mu\rightarrow \frac{4}{r_0^2} \left(1-\frac{2}{r_0}m e^{-t/2}\right)e^{-t} \label{unifAvgHEst2}
\end{align}
for every $t\in [0,T]$.

\end{Lem}
\begin{rmrk}
The corresponding Lemma holds in the hyperbolic setting for $m_H^{\Hy}$ and $\mathcal{H}^2$.
\end{rmrk}

By rearranging the Geroch monotonicity calculation we arrive at the following result.

\begin{Lem}\label{dtintEstimate}
For any solution of IMCF we have the following formula
\begin{align}
\frac{d}{dt} \int_{\Sigma_t} H^2 d\mu =  \frac{(16 \pi)^{3/2}}{|\Sigma_t|^{1/2}} \left (\frac{1}{2} m_H(\Sigma_t) - \frac{d}{dt}m_H(\Sigma_t) \right )
\end{align}
So if we assume that 
\begin{align}
m_H(\Sigma_t^i) \rightarrow 0\text{ as }i \rightarrow \infty
\end{align}
 then we have for a.e. $t \in [0,T]$ that 
\begin{align}
\frac{d}{dt} \int_{\Sigma_t^i} H^2 d\mu \rightarrow 0\label{Eq-dtH^2PMT}
\end{align}

If we assume that 
\begin{align}
m_H(\Sigma_T^i) - m_H(\Sigma_0^i) \rightarrow 0\text{ and }m_H(\Sigma_t^i)\rightarrow m > 0\text{ as } i \rightarrow \infty
\end{align}
 then we have that 
\begin{align}
\frac{d}{dt} \int_{\Sigma_t^i} H^2 d\mu \rightarrow \frac{16 \pi}{r_0}me^{-t/2} \label{Eq-dtH^2RPI}
\end{align}
\end{Lem}
\begin{rmrk}
The corresponding Lemma holds in the hyperbolic setting for $m_H^{\Hy}$ and $\mathcal{H}^2$.
\end{rmrk}

The two following Corollaries state the crucial estimates which get the rest of the results moving in the right direction. These convergence results follow from the Geroch monotonicity calculation.

\begin{Cor} \label{GoToZero}Let $\Sigma^i\subset M^i$ be a compact, connected surface with corresponding solution to IMCF $\Sigma_t^i$. If 
\begin{align}
m_H(\Sigma_0)\ge0 \text{ and }m_H(\Sigma^i_T) \rightarrow 0
\end{align}
 then for almost every $t \in [0,T]$, 
\begin{align}
&\int_{\Sigma_t^i} \frac{|\nabla H_i|^2}{H_i^2}d \mu \rightarrow 0, \hspace{1 cm} \int_{\Sigma_t^i} (\lambda_1^i-\lambda_2^i)^2d \mu \rightarrow 0,\hspace{1 cm} \int_{\Sigma_t^i} R^i d \mu \rightarrow 0,
\\ &\int_{\Sigma_t^i} Rc^i(\nu,\nu)d \mu \rightarrow 0, \hspace{.6 cm} \int_{\Sigma_t^i} K_{12}^id \mu \rightarrow 0, \hspace{2.1 cm} \int_{\Sigma_t^i} H_i^2 d\mu\rightarrow 16\pi,
\\&\int_{\Sigma_t^i} |A|_i^2 d \mu \rightarrow 8 \pi, \hspace{1 cm} \int_{\Sigma_t^i} \lambda_1^i\lambda_2^i d \mu \rightarrow 4\pi, \hspace{2 cm} \chi(\Sigma_t^i) \rightarrow 2,
\end{align}
as $i \rightarrow \infty$ where $K_{12}$ is the ambient sectional curvature tangent to $\Sigma_t$. As a consequence $\Sigma_t^i$ must eventually become topologically a sphere. 

If 
\begin{align}
\left (m_H(\Sigma^i_T)-m_H(\Sigma^i_0) \right ) \rightarrow 0\text{ where } m_H(\Sigma_0) \rightarrow m > 0
\end{align}
 then the first three integrals listed above $\rightarrow 0$  and for almost every $t \in [0,T]$ 
\begin{align}
 &\int_{\Sigma_t^i} H_i^2 d\mu\rightarrow 16 \pi \left (1 - \sqrt{\frac{16 \pi}{|\Sigma_0|}}me^{-t/2}  \right ),
\\& \int_{\Sigma_t^i} |A|_i^2 d \mu \rightarrow 8 \pi \left (1 - \sqrt{\frac{16 \pi}{|\Sigma_0|}}me^{-t/2}  \right ), \hspace{0.5 cm}
 \\& \int_{\Sigma_t^i} \lambda_1^i\lambda_2^i d \mu \rightarrow 4 \pi \left (1 - \sqrt{\frac{16 \pi}{|\Sigma_0|}}me^{-t/2}  \right ),\hspace{0.5 cm}
 \\&\int_{\Sigma_t^i}Rc^i(\nu,\nu)d\mu \rightarrow -\frac{8\pi}{r_0}m e^{-t/2},
 \\ &\int_{\Sigma_t^i}K_{12}^i d\mu \rightarrow -\frac{8\pi}{r_0}m e^{-t/2},\hspace{3 cm}\chi(\Sigma_t^i) \rightarrow 2.
\end{align}
As a consequence $\Sigma_t^i$ must eventually become topologically a sphere.
\end{Cor}

Now a similar corollary in the hyperbolic setting.

\begin{Cor} \label{GoToZeroHyperbolic}Let $\Sigma^i\subset M^i$ be a sequence of compact, connected surface with corresponding solution to IMCF $\Sigma_t^i$ where $R^i \ge -6$. 

If 
\begin{align}
m_H^{\Hy}(\Sigma_0^i)\ge0 \text{ and }m_H^{\Hy}(\Sigma^i_T) \rightarrow 0
\end{align}
  then for almost every $t \in [0,T]$,
\begin{align}
&\int_{\Sigma_t^i} \frac{|\nabla H_i|^2}{H_i^2}d \mu \rightarrow 0, \hspace{.8 cm} \int_{\Sigma_t^i} (\lambda_1^i-\lambda_2^i)^2d \mu \rightarrow 0,\hspace{.2 cm} \int_{\Sigma_t^i} R^i+6 d \mu \rightarrow 0,
\\&\int_{\Sigma_t^i} Rc^i(\nu,\nu)+2d \mu \rightarrow 0,  \int_{\Sigma_t^i} K_{12}^i+1d \mu \rightarrow 0, \hspace{.4 cm} \int_{\Sigma_t^i} H_i^2-4 d\mu\rightarrow 16\pi,
\\&\int_{\Sigma_t^i} |A_i|^2-2 d \mu \rightarrow 8 \pi,\hspace{.5 cm} \int_{\Sigma_t^i} \lambda_1^i\lambda_2^i-1 d \mu \rightarrow 4\pi, \hspace{.4 cm} \chi(\Sigma_t^i) \rightarrow 2,
\end{align}
as $i \rightarrow \infty$ where $K_{12}$ is the ambient sectional curvature tangent to $\Sigma_t$. As a consequence $\Sigma_t^i$ must eventually become topologically a sphere. 

If 
\begin{align}
\left (m_H^{\Hy}(\Sigma^i_T)-m_H^{\Hy}(\Sigma^i_0) \right ) \rightarrow 0\text{ where }m_H^{\Hy}(\Sigma_0) \rightarrow m > 0
\end{align}
 then the first three integrals listed above tend to zero and for almost every $t \in [0,T]$,
\begin{align}
 &\int_{\Sigma_t^i} H_i^2-4 d\mu\rightarrow 16 \pi \left (1 - \sqrt{\frac{16 \pi}{|\Sigma_0|}}me^{-t/2}  \right ),
 \hspace{0.2 cm} 
 \\&\int_{\Sigma_t^i} |A_i|^2-2 d \mu \rightarrow 8 \pi \left (1 - \sqrt{\frac{16 \pi}{|\Sigma_0|}}me^{-t/2}  \right ), \hspace{0.5 cm}
 \\& \int_{\Sigma_t^i} \lambda_1^i\lambda_2^i-1 d \mu \rightarrow 4 \pi \left (1 - \sqrt{\frac{16 \pi}{|\Sigma_0|}}me^{-t/2}  \right ),\hspace{0.5 cm}
 \\&\int_{\Sigma_t^i}Rc^i(\nu,\nu)+2d\mu \rightarrow -\frac{8\pi}{r_0}m e^{-t/2},
 \\ &\int_{\Sigma_t^i}K_{12}^i+1d\mu \rightarrow \frac{8\pi}{r_0}m e^{-t/2},\hspace{3 cm}\chi(\Sigma_t^i) \rightarrow 2.
\end{align}
 As a consequence $\Sigma_t^i$ must eventually become topologically a sphere.
\end{Cor}

In the following proposition an important diffeomorphism from $U_T^i$ to $\Sigma \times [0,T]$ is defined which is used throughout the rest of the paper to define and show $W^{1,2}$ convergence of $\hat{g}^i$ to the appropriate prototype space. We start by choosing an area preserving diffeomorphism 
\begin{align}
F_i:\Sigma_0^i \rightarrow S^2(r_0)
\end{align}
 which we know is well defined since we assume that 
 \begin{align}
 |\Sigma_0^i| = |S^2(r_0)| = 4 \pi r_0^2.
 \end{align}
  Then by the evolution of area under IMCF, $F_i$ automatically extends to a diffeomorphism 
  \begin{align}
  F_i(t): \Sigma_t \rightarrow S^2(r_0e^{t/2})
  \end{align}
   which defines a diffeomorphism from $U_T^i$ to $\Sigma\times [0,T]$ and is the coordinate system we will use throughout the rest of the paper.

\begin{Prop}\label{avgH}If $\Sigma_t^i$ is a sequence of IMCF solutions where 
\begin{align}
\int_{\Sigma_t^i} \frac{|\nabla H|^2}{H^2}&d \mu \rightarrow 0 \text{ as }i \rightarrow \infty, 
\\0 < H_0 \le &H(x,t) \le H_1 < \infty,
\\&|A|(x,t) \le A_0 < \infty
\end{align}
 then
\begin{align}
\int_{\Sigma_t^i} (H_i - \bar{H}_i)^2 d \mu \rightarrow 0
\end{align}
as $i \rightarrow \infty$ for almost every $t \in [0,T]$ where $\bar{H}_i = \dashint_{\Sigma_t^i}H_id \mu$. 

Let $d\mu_t^i$ be the volume form on $\Sigma$ w.r.t. $g^i(\cdot,t)$ then we can find a parameterization of $\Sigma_t$ so that 
\begin{align}
d\mu_t^i = r_0^2 e^t d\sigma
\end{align}
where $d\sigma$ is the standard volume form on the unit sphere.

Then for almost every $t \in [0,T]$ and almost every $x \in \Sigma$, with respect to $d\sigma$, we have that 
\begin{align}
H_i(x,t) - \bar{H}_i(t) \rightarrow  0,
\end{align}
along a subsequence. 

\end{Prop}
In Corollary \ref{GoToZero} and Corollary \ref{GoToZero2} we note that the Ricci curvature integrals are not so useful since we have not assumed anything about the sign of the Ricci curvature. In order to obtain useful estimates of the Ricci curvature we now turn to obtain weak convergence to the expected values.

\begin{Lem}\label{WeakRicciEstimate}Let $\Sigma^i_0\subset M^3_i$ be a compact, connected surface with corresponding solution to IMCF $\Sigma_t^i$. Then if $\phi \in C_c^1(\Sigma\times (a,b))$  and $0\le a <b\le T$ we can compute the  estimate
\begin{align}
\int_a^b\int_{\Sigma_t^i}& 2\phi Rc^i(\nu,\nu)d\mu d t=  \int_{\Sigma_a^i} \phi H_i^2 d\mu -  \int_{\Sigma_b^i} \phi H_i^2 d\mu 
\\&+ \int_a^b\int_{\Sigma_t^i}2\phi\frac{|\nabla H_i|^2}{H_i^2}-2\frac{\hat{g}^j(\nabla \phi, \nabla H_i)}{H_i} +\phi(H_i^2-2|A|_i^2) d\mu 
\end{align}

If 
\begin{align}
m_H(\Sigma^i_T) \rightarrow 0
\end{align}
 and $\Sigma_t$ satisfies the hypotheses of Proposition \ref{avgH} then  
\begin{align}
\int_a^b\int_{\Sigma_t^i} \phi Rc^i(\nu,\nu)d\mu dt \rightarrow 0
\end{align}
If 
\begin{align}
m_H(\Sigma^i_T)-m_H(\Sigma^i_0)  \rightarrow 0,&
\\m_H(\Sigma_T) \rightarrow m > 0,&
\end{align} 
and $\Sigma_t$ satisfies the hypotheses of Proposition \ref{avgH}  then 
\begin{align}
\int_a^b\int_{\Sigma_t}& \phi Rc^i(\nu,\nu)d\mu d t \rightarrow \int_a^b\int_{\Sigma_t} \frac{-2}{r_0}me^{-t/2} \phi d\mu dt.
\end{align}
\end{Lem}

Now a similar lemma in the hyperbolic setting.

\begin{Lem}\label{WeakRicciEstimateHyperbolic}Let $\Sigma^i_0\subset M^3_i$ be a compact, connected surface with corresponding solution to IMCF $\Sigma_t^i$. Then if $\phi \in C^1(\Sigma\times (a,b))$, $0\le a <b\le T$, and $\sigma$ is the round metric on $S^2$ with area element $d \sigma$ we can compute the estimate,
\begin{align}
\int_a^b\int_{\Sigma_t^i}& 2\phi Rc^i(\nu,\nu)d\mu d t=  \int_{\Sigma_a^i} \phi H_i^2 d\mu -  \int_{\Sigma_b^i} \phi H_i^2 d\mu 
\\&+ \int_a^b\int_{\Sigma_t^i}2\phi\frac{|\nabla H_i|^2}{H_i^2}-2\frac{\hat{g}^j(\nabla \phi, \nabla H_i)}{H_i} +\phi(H_i^2-2|A|_i^2) d\mu .
\end{align}

If 
\begin{align}
m_H^{\Hy}(\Sigma^i_T) \rightarrow 0
\end{align} 
and $\Sigma_t$ satisfies the hypotheses of Proposition \ref{avgH} then 
\begin{align}
\int_a^b\int_{\Sigma_t}& \phi Rc^i(\nu,\nu)d\mu d t \rightarrow  \int_a^b \int_{\Sigma} -2r_0^2 e^t \phi d\sigma dt.
\end{align}
If 
\begin{align}
m_H^{\Hy}(\Sigma^i_T)-m_H^{\Hy}(\Sigma^i_0)  \rightarrow 0,&
\\m_H^{\Hy}(\Sigma_T) \rightarrow m > 0,&
\end{align}
 and $\Sigma_t$ satisfies the hypotheses of Proposition \ref{avgH} then
\begin{align}
\int_a^b\int_{\Sigma_t}& \phi Rc^i(\nu,\nu)d\mu d t \rightarrow \int_a^b\int_{\Sigma} -2\left(\frac{1 }{r_0}me^{-t/2} + r_0^2 e^t\right) \phi d\sigma dt.
\end{align}
\end{Lem}

We end this subsection with a lemma which allows us to control the metric on $\Sigma_t^i$ in terms of the metric on $\Sigma_T^i$.

\begin{Lem} \label{metricEst}Assume that $\Sigma_t^i$ is a solution to IMCF and let
\begin{align}
\lambda_1^i(x,t)\le \lambda_2^i(x,t)
\end{align}
 be the eigenvalues of $A^i(x,t)$ then 
\begin{align}
 e^{\int_T^t\frac{2\lambda^i_1(x,s)}{H^i(x,s)}ds} g^i(x,T) \le g^i(x,t) &\le e^{\int_T^t\frac{2\lambda^i_1(x,s)}{H^i(x,s)}ds} g^i(x,T)
\end{align}
\end{Lem}

\subsection{New Estimates for $W^{1,2}$ Convergence}\label{subsec: New Estimates}

In this section we prove new estimates which are in particular useful for proving $W^{1,2}$ convergence. This $W^{1,2}$ convergence will be defined with respect to $(\Sigma\times [0,T], \delta)$ and hence we are concerned with derivatives with respect to the polar coordinates defined in \eqref{deltaMetricDef} with the coordinate vectors $\{\partial_0=\partial_t,\partial_1,\partial_2\}$.

We begin by deriving an equation for the evolution of the average of $H$ under IMCF.

\begin{Lem}\label{AvgIntProp}
If we let $\Sigma_t$ be a solution of IMCF and define 
\begin{align}
\bar{H} = \frac{1}{|\Sigma_t|} \int_{\Sigma_t} H d \mu = \dashint_{\Sigma_t}H d \mu 
\end{align}
then
\begin{align}
\frac{d \bar{H}}{d t} &=\dashint_{\Sigma_t}\frac{\partial H}{\partial t} d\mu.
\end{align}
\end{Lem}
\begin{proof} 
\begin{align}
\frac{\partial \bar{H}_i}{\partial t} &= \frac{\partial}{\partial t} \left (|\Sigma_t|^{-1} \int_{\Sigma_t} H_i dV \right ) 
\\&= -|\Sigma_t|^{-1} \int_{\Sigma_t} H_i dV + |\Sigma_t|^{-1} \int_{\Sigma_t} \frac{\partial H_i}{\partial t} dV + |\Sigma_t|^{-1} \int_{\Sigma_t} H_i dV
\\ &=\dashint_{\Sigma_t}\frac{\partial H_i}{\partial t} dV.
\end{align}
\end{proof}

Now we can use Lemma \ref{AvgIntProp} to derive a more specific equation.

\begin{Lem}\label{AvgIntEvol}
If we let $\Sigma_t$ be a solution of IMCF and define 
\begin{align}
\bar{H} = \frac{1}{|\Sigma_t|} \int_{\Sigma_t} H d \mu = \dashint_{\Sigma_t}H d \mu 
\end{align}
then
\begin{align}
\frac{d \bar{H}}{d t} &=-\dashint_{\Sigma_t}\left(\frac{|A|^2 + Rc(\nu,\nu)}{H}\right ) d\mu.
\end{align}
\end{Lem}
\begin{proof}
The evolution equation for $H$ under IMCF is given by
\begin{align}
\frac{\partial H}{\partial t} = -\Delta \left (\frac{1}{H} \right ) -\frac{|A|^2}{H} - \frac{Rc(\nu,\nu)}{H}
\end{align}
and so if we take the average integral of both sides of this equation we find
\begin{align}
\dashint_{\Sigma_t}\frac{\partial H}{\partial t}d\mu = -\dashint_{\Sigma_t}\Delta \left (\frac{1}{H} \right )d\mu -\dashint_{\Sigma_t}\left(\frac{|A|^2}{H} + \frac{Rc(\nu,\nu)}{H}\right )d\mu
\end{align}
and so by the divergence theorem and Lemma \ref{AvgIntProp} we find
\begin{align}
\frac{d }{d t}\dashint_{\Sigma_t}Hd\mu =  -\dashint_{\Sigma_t}\left(\frac{|A|^2}{H} + \frac{Rc(\nu,\nu)}{H}\right )d\mu.
\end{align}
\end{proof}

Similarly, we can derive an equation for the average of $H^2$.

\begin{Lem}\label{AvgIntEvol2}
If we let $\Sigma_t$ be a solution of IMCF and define 
\begin{align}
\bar{H}^2 = \frac{1}{|\Sigma_t|} \int_{\Sigma_t} H^2 d \mu = \dashint_{\Sigma_t}H^2 d \mu 
\end{align}
then
\begin{align}
\frac{d \bar{H}^2}{d t} &=-2\dashint_{\Sigma_t}\left (2\frac{|\nabla H|^2}{H^2} +|A|^2 + Rc(\nu,\nu)\right) d\mu.
\end{align}
\end{Lem}
\begin{proof}
The evolution equation for $H$ under IMCF is given by
\begin{align}
\frac{\partial H^2}{\partial t} = -2H\Delta \left (\frac{1}{H} \right )-2\frac{|\nabla H|^2}{H^2} -2|A|^2 -2 Rc(\nu,\nu)
\end{align}
and so if we take the average integral of both sides of this equation we find
\begin{align}
\dashint_{\Sigma_t}\frac{\partial H^2}{\partial t}d\mu = -2\dashint_{\Sigma_t}\left(H\Delta \left (\frac{1}{H} \right ) + \frac{|\nabla H|^2}{H^2}\right)d\mu -2\dashint_{\Sigma_t}\left(|A|^2 + Rc(\nu,\nu)\right )d\mu
\end{align}
and so by integration by parts and Lemma \ref{AvgIntProp} we find
\begin{align}
\frac{d }{d t}\dashint_{\Sigma_t}H^2d\mu =  -2\dashint_{\Sigma_t}\left( 2\frac{|\nabla H|^2}{H^2}+|A|^2 + Rc(\nu,\nu)\right )d\mu.
\end{align}
\end{proof}

We now use the previous lemmas to deduce what the evolution of the average mean curvature must converge to.

\begin{Cor}
\label{GoToZero2}Let $\Sigma^i\subset M^i$ be a compact, connected surface with corresponding solution to IMCF $\Sigma_t^i$. If 
\begin{align}
&m_H(\Sigma_0)\ge0\text{ and }m_H(\Sigma^i_T) \rightarrow 0
\\ \text{ or }&m_H^{\Hy}(\Sigma_0)\ge0\text{ and }m_H^{\Hy}(\Sigma^i_T) \rightarrow 0
\end{align}
  then for almost every $t \in [0,T]$ and almost every $x \in \Sigma_t$ we have that
\begin{align}
\frac{\partial \bar{H}_i^2}{\partial t} &\rightarrow \frac{-4}{r_0^2} e^{-t} \hspace{3cm} \frac{\partial \bar{H}_i}{\partial t} \rightarrow -\frac{e^{-t/2}}{r_0}
\end{align}
If 
\begin{align}
&\left (m_H(\Sigma^i_T)-m_H(\Sigma^i_0) \right ) \rightarrow 0\text{ where }m_H(\Sigma_0) \rightarrow m > 0
\\ \text{ or } &\left (m_H^{\Hy}(\Sigma^i_T)-m_H^{\Hy}(\Sigma^i_0) \right ) \rightarrow 0\text{ where }m_H^{\Hy}(\Sigma_0) \rightarrow m > 0
\end{align}
then for almost every $t \in [0,T]$ and almost every $x\in \Sigma_t$ we have that
\begin{align}
\frac{\partial \bar{H}_i^2}{\partial t} &\rightarrow \frac{-4}{r_0^2}\left (1- \frac{2m}{r_0} e^{-t/2} \right ) e^{-t} \hspace{1cm} \frac{\partial \bar{H}_i}{\partial t} \rightarrow -\frac{1}{r_0}\sqrt{1- \frac{2m}{r_0} e^{-t/2} } e^{-t/2}\hspace{1cm} 
\end{align}
\end{Cor}
\begin{proof}
This follows by combining Lemma \ref{AvgIntEvol} and Lemma \ref{AvgIntEvol2} with Corollary \ref{GoToZero}.
\end{proof}

In the following lemma we obtain estimates on the derivatives in the $\Sigma$ direction in the coordinate space $\Sigma \times [0,T]$.

\begin{Lem}\label{DerivativesOfMetricEstimates}
We can find the following estimates on the coordinate derivatives of the metric $g$
\begin{align}
\left|Dg(x,T)-Dg(x,t)\right| &\le \int_t^T \frac{|D A|_i}{H_i} + \frac{|A|_i|D H_i|}{H_i^2}  ds
\\\left|Dg(x,T)-e^{t-T}Dg(x,t)\right| &\le \int_t^Te^{s-T} \lp \frac{|D A|_i}{H_i} + \frac{|A|_i|D H_i|}{H_i^2}+  |Dg| \rp ds,
\end{align}
where $D$ is the covariant derivative with respect to $\sigma$ and all norms are taken with respect to $\sigma$.
\end{Lem}

\begin{proof}
We start by taking spatial derivatives of the equation,
\begin{align}
\frac{\partial g_{lm}^i}{\partial t} = \frac{A_{lm}^i}{H_i},
\end{align}
 in normal coordinates with respect to $\sigma$ centered at $x$,
\begin{align}
\frac{\partial}{\partial t}  g_{lm,k} &=  \frac{A_{lm,k}^i}{H_i} - \frac{A_{lm}^i}{H_i^2}H_{i,k}
\\\frac{\partial}{\partial t} (e^{t-T} g_{lm,k}) &= e^{t-T} \lp \frac{A_{lm,k}}{H_i} - \frac{A_{lm}}{H_i^2}H_{i,k}\rp + e^{t-T} g_{lm,k}.
\end{align}
Now by taking norms with respect to $\sigma$ of both sides yields the inequality,
\begin{align}
\left |\frac{\partial}{\partial t}(e^{t-T} Dg)\right|&\le  e^{t-T} \lp \frac{|D A|_i}{H_i} +  \frac{|A|_i|D H_i|}{H_i^2} +  |Dg|\rp, 
\end{align}
where $D$ is the covariant derivative with respect to $\sigma$.

Now to finish up we find,
\begin{align}
&\left|Dg(x,t)-e^{t-T}Dg(x,0)\right| =\left | \int_0^t \frac{\partial}{\partial s}(e^{s-T}Dg)ds  \right |
\\&\le\int_0^t \left | \frac{\partial}{\partial s}(e^{s-T}Dg) \right |ds 
\\&\le \int_0^t e^{s-T} \lp \frac{|D A|_i}{H_i} +  \frac{|A|_i|D H_i|}{H_i^2} +  |Dg|\rp ds,
\end{align}
where all norms are taken with respect to $\sigma$, which yields the second estimate and the first estimate follows similarly.
\end{proof}

In order for the previous lemma to be useful we will need to deduce integral estimates for $|\nabla A|$. The following interpolation inequality will be key which can be found in section 12 of the work of Hamilton \cite{Ha}.

\begin{Lem}\label{InterpolationInequality}
If $T$ is a tensor on $\Sigma$ then there exists a constant $C(n,m)$, independent of the metric and the connection, so that the following estimate holds
\begin{align}
\int_{\Sigma}|\nabla^iT|^2 d \mu \le C \left (\int_{\Sigma}|\nabla^mT|^2 d \mu \right )^{i/m}\left (\int_{\Sigma} |T|^2 \right )^{1-i/m}
\end{align}
for $0\le i \le m$.
\end{Lem}

We now use this interpolation inequality in combination with the assumptions of Definition \ref{IMCFClass} to show the desired integral convergence of $|D A|$.

\begin{Cor}\label{2ndFundFormDerToZero}
Let $\Sigma^i\subset M^i$ be a compact, connected surface with corresponding solution to IMCF $\Sigma_t^i$ such that
\begin{align}
\|A_i\|_{W^{2,2}(\Sigma\times[0,T])} \le C.
\end{align} 
If 
\begin{align}
&m_H(\Sigma_0)\ge0\text{ and } m_H(\Sigma^i_T) \rightarrow 0
\\ \text{ or }&\left (m_H(\Sigma^i_T)-m_H(\Sigma^i_0) \right ) \rightarrow 0\text{ where }m_H(\Sigma_0) \rightarrow m > 0
\end{align}
 then for almost every $t \in [0,T]$ 
\begin{align}
\int_{\Sigma} |D A_i|^2 r_0^2e^td \sigma \rightarrow 0.
\end{align}
\end{Cor}
\begin{proof}
Consider the tensor $T=A-\frac{e^{-t/2}}{r_0}g$ and apply Lemma \ref{InterpolationInequality} with $m=2$ and $i=1$ to find
\begin{align}
\int_{\Sigma_t^i}|D A|^2 d \mu &\le C \left (\int_{\Sigma_t^i}|D^2 A|^2 d \mu \right )^{1/2}\left (\int_{\Sigma_t^i} |A-\frac{e^{-t/2}}{r_0}g|^2 \right )^{1/2}
\\&\le C \sqrt{A_2} \left (\int_{\Sigma_t^i} n \max \left\{|\lambda_1-\frac{e^{-t/2}}{r_0}|^2,|\lambda_2-\frac{e^{-t/2}}{r_0}|^2 \right\}\right )^{1/2} \rightarrow 0
\end{align}
where $\lambda_1,\lambda_2$ are the eigenvalues of $A$.
The last convergence result follows from the definition of $|\nabla A|$ and $|DA|$ as well as the assume $L^2$ convergence.
\end{proof}

Now we will prove estimates which allow us to use the evolution equation for $H$ to gain $W^{1,2}$ control on $H$ at the price of assuming $L^2$ or $W^{1,2}$ control on the Ricci curvature.

\begin{Lem}\label{RicEst} Let $\Sigma_t$ be a solution of IMCF such that
\begin{align} 
0 < H_0 \le H(x,t) \le H_1 < \infty
\end{align}
 then 
\begin{align}
&\int_0^T\int_{\Sigma_t} 4|Rc(\nu,\nu)|^2 +8|A|^2|Rc(\nu,\nu)| + 4 |A|^4d\mu dt  + \int_{\Sigma_0} \frac{|\nabla H|^2}{H^2}d \mu 
\\&\ge \int_0^T\int_{\Sigma_t} \left (\frac{\partial H^2}{\partial t}\right )^2 + \frac{(\Delta H^2)^2}{H^4}d\mu dt + \sup_{t \in[0,T]} \int_{\Sigma_t}\frac{|\nabla H|^2}{H^2} d\mu.
\end{align}
If the stronger estimate of the hessian of $H^2$ is needed instead of just the laplacian then the above estimate can be improved to find 
\begin{align}
\int_0^T\int_{\Sigma_t} 4|Rc(\nu,\nu)|^2& +8|A|^2|Rc(\nu,\nu)|+4|A|^4+ |Rc| \frac{|\nabla H^2|^2}{H_1^4}d\mu dt 
\\+ \int_{\Sigma_0} \frac{|\nabla H|^2}{H^2}d \mu  &\ge \int_0^T\int_{\Sigma_t} \left (\frac{\partial H^2}{\partial t}\right )^2 + \frac{|\nabla \nabla H^2|^2}{H_1^4}d\mu dt 
\\&+ \sup_{t \in[0,T]} \int_{\Sigma_t}\frac{|\nabla H|^2}{H^2} d\mu.
\end{align}

\end{Lem}
\begin{proof}
We start by integrating the square of the following linear PDE for $H^2$
\begin{align}
\left(\partial_t - \frac{\Delta}{H^2} \right )H^2 = -2|A|^2-2Rc(\nu,\nu)
\end{align}
in order to find
\begin{align}
\int_{\Sigma_t}(-2|A|^2-2Rc(\nu,\nu))^2 d \mu &= \int_{\Sigma_t}\left [\left(\partial_t - \frac{\Delta}{H^2} \right )H^2 \right ]^2 d \mu
\end{align}
from which we obtain
\begin{align}
&\int_{\Sigma_t}4|A|^4+8|A|^2|Rc(\nu,\nu)|+ 4|Rc(\nu,\nu))|^2 d \mu 
\\&\ge \int_{\Sigma_t}\left ( \frac{\partial H^2}{\partial t} \right )^2 - 2 \frac{\partial H^2}{\partial t}\frac{\Delta H^2}{H^2} + \frac{(\Delta H^2)^2}{H^4} d \mu
\\&\ge \int_{\Sigma_t}\left ( \frac{\partial H^2}{\partial t} \right )^2 +2\frac{\partial }{\partial t}\nabla_k  H^2 \frac{\nabla^kH^2}{H^2}- 2 \frac{\partial H^2}{\partial t}\frac{|\nabla H^2|^2}{H^3} +\frac{(\Delta H^2)^2}{H_1^4} d \mu 
\\&= \int_{\Sigma_t}\left ( \frac{\partial H^2}{\partial t} \right )^2 + \frac{\partial }{\partial t} \left (\frac{|\nabla H^2|}{H^2} \right ) -\frac{\nabla^kH^2 \nabla_k\nabla ^i\nabla_iH^2}{H_1^4} d \mu 
\\&= \int_{\Sigma_t}\left ( \frac{\partial H^2}{\partial t} \right )^2 + \frac{\partial }{\partial t} \left (\frac{|\nabla H^2|}{H^2} \right ) -\frac{\nabla^kH^2 \nabla_i\nabla ^i\nabla_kH^2}{H_1^4} -\frac{R_{kl} \nabla^kH^2\nabla^lH^2}{H_1^4} d \mu 
\\&= \int_{\Sigma_t}\left ( \frac{\partial H^2}{\partial t} \right )^2 + \frac{\partial }{\partial t} \left (\frac{|\nabla H^2|}{H^2} \right ) +\frac{|\nabla \nabla H^2|^2}{H_1^4} -\frac{Rc(\nabla H^2,\nabla H^2)}{H_1^4} d \mu
\end{align}
So now we integrate from $0$ to $t'$ with respect to $t$, $t' \in [0,T]$, to find
\begin{align}
&\int_0^{t'}\int_{\Sigma_t} 4|Rc(\nu,\nu)|^2+8|A|^2|Rc(\nu,\nu)|+4|A|^4 +  \frac{|Rc||\nabla H^2|^2}{H_1^4}d\mu dt 
\\&\ge \int_0^{t'}\int_{\Sigma_t} 4|Rc(\nu,\nu)|^2 +8|A|^2|Rc(\nu,\nu)|+4|A|^4+  \frac{Rc(\nabla H^2, \nabla H^2)}{H_1^4}d\mu dt 
\\&\ge \int_0^{t'}\int_{\Sigma_t} \left (\frac{\partial H^2}{\partial t}\right )^2 + \frac{|\nabla \nabla H^2|^2}{H_0^4}d\mu dt +  \int_{\Sigma_{t'}}\frac{|\nabla H|^2}{H^2} d\mu - \int_{\Sigma_0}\frac{|\nabla H|^2}{H^2} d\mu
\end{align}
and then taking the $\sup_{t' \in [0,T]}$ of both sides we find the desired estimate. 
\end{proof}

We now show how to use the previous Lemma to show new higher order convergence results.

\begin{Lem}\label{dHdtConvergence}
If $\Sigma_t^i$ is a sequence of IMCF solutions where 
\begin{align}
&m_H(\Sigma_t) \rightarrow 0\text{ as }i \rightarrow \infty,
\\&m_H^{\Hy}(\Sigma_t) \rightarrow 0\text{ as }i \rightarrow \infty,
\\ &m_H(\Sigma_{T}^i)- m_H(\Sigma_{0}^i) \rightarrow 0\text{ and }m_H(\Sigma_t^i)\rightarrow m > 0,
\\ \text{ or }&m_H^{\Hy}(\Sigma_{T}^i)- m_H^{\Hy}(\Sigma_{0}^i) \rightarrow 0\text{ and }m_H^{\Hy}(\Sigma_t^i)\rightarrow m > 0,
\end{align}
and
\begin{align}
 0 < H_0 \le H(x,t) &\le H_1 < \infty, 
\\|A|(x,t) &\le A_0 < \infty,
\\\|Rc^i(\nu,\nu)\|_{W^{1,2}(\Sigma \times [0,T])} &\le C \label{SobolevRicciBound}
\end{align}
then
\begin{align}
\int_0^T\int_{\Sigma_t} \left (\frac{\partial H_i^2}{\partial t}\right )^2-4|A|_i^4 d \mu dt &\rightarrow 0,
 \\ \int_0^T\int_{\Sigma_t} \frac{(\Delta H_i^2)^2}{H_i^4} d \mu dt &\rightarrow 0,
\\ \sup_{t \in[0,T]} \int_{\Sigma_t}\frac{|\nabla H_i|^2}{H_i^2} d\mu  &\rightarrow 0.
\end{align}
\end{Lem}
\begin{proof}
Notice that by combining \eqref{SobolevRicciBound} with Lemma \ref{WeakRicciEstimate} we find that
\begin{align}
\|Rc^i(\nu,\nu)\|_{L^2(\Sigma \times [0,T])} \rightarrow 0. \label{L2RicciConvergence}
\end{align}
Then we note that Lemma \ref{RicEst} implies
\begin{align}
&\int_0^T\int_{\Sigma_t} 4|Rc^i(\nu,\nu)|^2 +8|A|_i^2|Rc^i(\nu,\nu)| d\mu dt  + \int_{\Sigma_0} \frac{|\nabla H_i|^2}{H_i^2}d \mu 
\\&\ge \int_0^T\int_{\Sigma_t} \left (\frac{\partial H_i^2}{\partial t}\right )^2- 4 |A|_i^4 + \frac{(\Delta H_i^2)^2}{H_i^4}d\mu dt + \sup_{t \in[0,T]} \int_{\Sigma_t}\frac{|\nabla H_i|^2}{H_i^2} d\mu
\end{align}
and \eqref{L2RicciConvergence} implies
\begin{align}
&\int_0^T\int_{\Sigma_t}|A|_i^2|Rc^i(\nu,\nu)| d \mu dt 
\\&\le \lp\int_0^T\int_{\Sigma_t}|A|_i^4 d \mu dt\rp^{1/2} \lp\int_0^T\int_{\Sigma_t}|Rc^i(\nu,\nu)|^2 d \mu dt\rp^{1/2} \rightarrow 0.
\end{align}
Now since $\frac{\partial \bar{H}_i^2}{\partial t} = \dashint_{\Sigma_t}\frac{\partial H_i^2}{\partial t} d \mu$  we notice that Corollary \ref{GoToZero2} implies
\begin{align}
\lim_{i \rightarrow \infty} \left[\left (\frac{\partial H_i^2}{\partial t}\right )^2-4|A|_i^4 \right]\ge 0,\label{aeConvergence}
\end{align}
for a.e. $t \in [0,T]$ and $x \in \Sigma$. Combining \eqref{aeConvergence} with Fatou's Lemma  and the assumption $|A|_i \le A_0$ we find,
\begin{align}
 \int_0^T\int_{\Sigma_t} \left (\frac{\partial H_i^2}{\partial t}\right )^2-4|A|_i^4 d \mu dt &\rightarrow 0,
 \\ \int_0^T\int_{\Sigma_t} \frac{(\Delta H_i^2)^2}{H_i^4} d \mu dt &\rightarrow 0,
 \\ \sup_{t \in[0,T]} \int_{\Sigma_t}\frac{|\nabla H_i|^2}{H_i^2} d\mu  &\rightarrow 0.
\end{align}
\end{proof}

\subsection{Consequences of Rigidity Results} \label{subsec: Rigidity}

In this subsection we use rigidity results for Riemannian manifolds to deduce important consequences for the metric on $\Sigma_t$, $g^i(x,t)$.

\begin{Thm}(Theorem 6.4 of Petersen \cite{P}) \label{rigidity2}
Let $n \ge 2$ $\Lambda, v, D > 0$ and $c \in \R$ be given. There exists $\epsilon=\epsilon(n,\Lambda,v, D)$ such that any $(\Sigma,g)$ satisfying 
\begin{align}
&|Rc| \le \Lambda
\\& vol(\Sigma)\ge v
\\& diam(\Sigma) \le D
\\&  \left (\int_{\Sigma} \|Rm - \lambda g \circ g\|^{n/2} d \mu \right )^{2/n}\le \epsilon
\end{align}
is $C^{1+\alpha}$ close to a constant curvature metric for any $\alpha < 1$.
\end{Thm}
Note that when $n=2$ the Riemann curvature tensor is $Rm = K g \circ g$, where $g \circ g$ represents the Kulkarni-Nomizu product, and so $\|Rm - \lambda g \circ g\|^2 = \|g \circ g\|^2 |K - \lambda|^2 = 2^4|K - \lambda|^2$.

\begin{Lem}\label{C1alphaConvergence}
Let $\Sigma_t^i$ be a sequence of solutions to IMCF such that $\Sigma_T^i$ is a sequence of hypersurfaces such that 
\begin{align}
&m_H(\Sigma_t) \rightarrow 0\text{ as }i \rightarrow \infty,
\\&m_H^{\Hy}(\Sigma_t) \rightarrow 0\text{ as }i \rightarrow \infty,
\\ &m_H(\Sigma_{T}^i)- m_H(\Sigma_{0}^i) \rightarrow 0\text{ and }m_H(\Sigma_t^i)\rightarrow m > 0,
\\ \text{ or }&m_H^{\Hy}(\Sigma_{T}^i)- m_H^{\Hy}(\Sigma_{0}^i) \rightarrow 0\text{ and }m_H^{\Hy}(\Sigma_t^i)\rightarrow m > 0,
\end{align}  
If we assume,
\begin{align}
&\| Rc^i(\nu,\nu)\|_{W^{1,2}(\Sigma\times \{T\})} \le C,
\\&\| R^i\|_{L^2(\Sigma\times \{T\})} \le C,
\\&diam(\Sigma_T^i) \le D \text{ } \forall \text{ } i,
\\&  |K^i| \le C \text{ on } \Sigma_T,
\end{align}
then we find that
\begin{align}
g^i(x,T) \rightarrow r_0^2 e^T\sigma(x)
\end{align}
in $C^{1,\alpha}$, $0 <\alpha < 1$.
\end{Lem}
\begin{proof}
By the assumption that $\|Rc^i(\nu,\nu)\|_{W^{1,2}(\Sigma\times \{T\})} \le C$ we know by Sobolev embedding that a subsequence converges strongly in $L^2(\Sigma\times\{T\})$ with the measure $r_0^2d\sigma$, i.e. 
\begin{align}\label{RicciL2Converges}
\int_{\Sigma}|Rc^j(\nu,\nu)|^2 r_0^2d\sigma \rightarrow k(x,t) \in L^2(\Sigma\times\{T\}).
\end{align}
 By combining \eqref{RicciL2Converges} with Lemma \ref{WeakRicciEstimate} we find that 
 \begin{align}
 \int_{\Sigma}|Rc^j(\nu,\nu)|^2 r_0^2d\sigma  \rightarrow 0.
 \end{align} 
 Then since $R^i$ converges to $0$ in $L^1$ by Corollary \ref{GoToZero} we can combine with the assumed $L^2$ bound and an interpolation inequality to conclude that $R^i$ converges to $0$ in $L^2$.  Then we have that $\int_{\Sigma}|K_{12}^i| d\mu^{\infty}_0  \rightarrow 0$ and hence 
\begin{align}
\int_{\Sigma}|K^i-\frac{1}{r_0^2}|r_0^2d\sigma &=\int_{\Sigma}|K_{12}^i+ \lambda_1^i\lambda_2^i - \frac{1}{r_0^2}|r_0^2d\sigma
\\&\le 2\int_{\Sigma}|K_{12}^i|+|\lambda_1^i\lambda_2^i - \frac{1}{r_0^2}| r_0^2d\sigma \rightarrow 0 \label{Eq-lastEnd1}
\end{align}
where we use the pointwise a.e. convergence of $\lambda_j^i$, $j=1,2$, that is implied by Corollary \ref{GoToZero}, on a subsequence, and the bound $|A|_i\le C$.

This shows that $\int_{\Sigma} |K^i - \frac{1}{r_0^2}| d \mu_0^{\infty} \rightarrow 0$ and hence by combining with the diameter bound diam$(\Sigma_0^i)\le D$ and the pointwise Gauss curvature bound $K^i \le C$ on $\Sigma_T$ then  we can apply the rigidity result of Petersen  \cite{P}, Corollary \ref{rigidity2}, which implies that $|g^i(x,0) - r_0^2\sigma(x)|_{C^{1+\alpha}} \rightarrow 0$ as $i \rightarrow \infty$ where $\alpha < 1$. 

\end{proof}
\begin{rmrk}
By combining Lemma \ref{C1alphaConvergence} with Lemma \ref{metricEst} and the bounds assumed in Definition \ref{IMCFClass} we now have shown that $\hat{g}^i$ has a uniform upper bound. Note that this implies that the metrics $g^i$ and $\sigma$ are uniformly equivalent on $\Sigma_t$ for $t \in [0,T]$ and hence quantities which are converging to zero in coordinates will converge to zero in norm with respect to either metric.
\end{rmrk}

\begin{Thm}(Corollary 1.5 of Petersen-Wei \cite{PW})\label{rigidity}
Given any integer $n \ge 2$, and numbers $p > n/2$, $\lambda \in \R$, $v >0$, $D < \infty$, one can find $\epsilon = \epsilon(n,p,\lambda, D) > 0$ such that a closed Riemannian $n-$manifold $(\Sigma,g)$ with,
\begin{align}
&\text{vol}(\Sigma)\ge v,
\\&\text{diam}(\Sigma) \le D,
\\& \frac{1}{|\Sigma|} \int_{\Sigma} \|Rm - \lambda g \circ g\|^p d \mu \le \epsilon(n,p,\lambda,D),\label{Eq-l2curv}
\end{align}
is $C^{\alpha}$, $\alpha < 2 - \frac{n}{p}$, close to a constant curvature metric on $\Sigma$.
\end{Thm}

\begin{Lem}\label{CalphaConvergence}
Let $\Sigma_t^i$ be a sequence of solutions to IMCF such that 
\begin{align}
&m_H(\Sigma_t) \rightarrow 0\text{ as }i \rightarrow \infty,
\\&m_H^{\Hy}(\Sigma_t) \rightarrow 0\text{ as }i \rightarrow \infty,
\\ &m_H(\Sigma_{T}^i)- m_H(\Sigma_{0}^i) \rightarrow 0\text{ and }m_H(\Sigma_t^i)\rightarrow m > 0,
\\ \text{ or }&m_H^{\Hy}(\Sigma_{T}^i)- m_H^{\Hy}(\Sigma_{0}^i) \rightarrow 0\text{ and }m_H^{\Hy}(\Sigma_t^i)\rightarrow m > 0,
\end{align}
If we assume,
\begin{align}
\|Rc^i(\nu,\nu)\|_{W^{1,2}(\Sigma\times[0,T])} &\le C,\label{RicciEstC0Proof}
\\\|R^i\|_{L^2(\Sigma\times[0,T])} &\le C,\label{ScalarEstC0Proof}
\\ Diam(\Sigma_t) &\le C \text{ for } t \in [0,T],\label{DiamEstC0Proof}
\end{align}
then on a subsequence,
\begin{align}
g^k(x,t) \rightarrow r_0^2 e^t \sigma(x),
\end{align}
in $C^{\alpha}$ for  a.e. $t \in [0,T]$.
\end{Lem}
\begin{proof}
By \eqref{RicciEstC0Proof} combined with Lemma \ref{WeakRicciEstimate} we find that
\begin{align}
\int_0^T\int_{\Sigma}|Rc^i(\nu,\nu)|^2 d \mu dt \rightarrow 0.
\end{align}
By combining \eqref{ScalarEstC0Proof} with Corollary \ref{GoToZero} and an interpolation inequality we find that
\begin{align}
\int_0^T\int_{\Sigma}|R^i|^p d \mu dt \rightarrow 0
\end{align}
for $1<p<2$ and hence
\begin{align}
\int_0^T\int_{\Sigma}|K_{12}^i|^p d \mu dt \le \int_0^T\int_{\Sigma}|Rc^i(\nu,\nu)|^p+|R^i|^p d \mu dt \rightarrow 0
\end{align}
for $1 < p < 2$. Now we notice that by combining with Corollary \ref{GoToZero} we find
\begin{align}
\int_0^T\int_{\Sigma}|K^i- \frac{1}{r_0^2}|^p d \mu dt \le \int_0^T\int_{\Sigma}|K_{12}^i|^p+|\lambda_1^i\lambda_2^i - \frac{1}{r_0^2}|^p d \mu dt \rightarrow 0
\end{align}
for $1<p<2$.
Now by combining with the rigidity result Theorem \ref{rigidity} and the diameter bound \eqref{DiamEstC0Proof} we find that on a subsequence
\begin{align}
g^k(x,t) \rightarrow r_0^2 e^t \sigma(x)
\end{align}
in $C^{\alpha}$ for a.e. $t \in [0,T]$.
\end{proof}

\section{Convergence To Prototype Spaces} \label{Sect-Conv}

In this section we successively show the pairwise convergence of interpolating metrics in $W^{1,2}$ from $\hat{g}^i(x,t)$ to $\delta$, $g_S$, $g_{\Hy}$, or $g_{ADSS}$. By combining all the pairwise convergence results using the triangle inequality we will be able to prove the main theorems of this work. In this section derivatives with respect to $\{\partial_0 = \partial_t,\partial_1,\partial_2\}$ denote Euclidean covariant derivatives with respect to polar coordinates on $\Sigma \times [0,T]$.

\begin{Thm}\label{gtog1} Let $U_{T}^i \subset M_i^3$ be a sequence such that $U_{T}^i\subset \mathcal{M}_{r_0,H_0,I_0}^{T,H_1,A_1}$ and assume
\begin{align}
&m_H(\Sigma_t) \rightarrow 0\text{ as }i \rightarrow \infty,
\\&m_H^{\Hy}(\Sigma_t) \rightarrow 0\text{ as }i \rightarrow \infty,
\\ &m_H(\Sigma_{T}^i)- m_H(\Sigma_{0}^i) \rightarrow 0\text{ and }m_H(\Sigma_t^i)\rightarrow m > 0,
\\ \text{ or }&m_H^{\Hy}(\Sigma_{T}^i)- m_H^{\Hy}(\Sigma_{0}^i) \rightarrow 0\text{ and }m_H^{\Hy}(\Sigma_t^i)\rightarrow m > 0.
\end{align}
 If we consider \eqref{AmbientMetricDef} and define the metric,
\begin{align}
g^i_1(x,t)&= \frac{1}{\overline{H}_i(t)^2}dt^2 + g^i(x,t),
\end{align}
 on $U_T^i$ then,
\begin{align}
\int_{U_T}|\hat{g}^i -g^i_1|^2 dV+ \sum_{k=0}^2\int_{U_T}|\partial_k \hat{g}^i -\partial_k g^i_1|^2 dV \rightarrow 0, 
\end{align}
 as $i \rightarrow \infty$ where $dV$ is the volume form on $U_T=\Sigma \times [0,T]$ with respect to $\delta$. Where $\{\partial_0 = \partial_t,\partial_1,\partial_2\}$ denotes the derivatives with respect to the coordinates on $\Sigma \times [0,T]$.
\end{Thm}

\begin{proof}
We have previously shown that,
\begin{align}
\int_{U_T}|\hat{g}^i -g^i_1|^2 dV \rightarrow 0, 
\end{align}
so now we concentrate on the derivative terms. We denote $\partial_t = \partial_0$ and then $\partial_1,\partial_2$ are the two spacial derivatives with respect to $\Sigma$. We can compute,
\begin{align}
\partial_k \hat{g}^i&=\frac{-2}{H_i(x,t)^3} \partial_k H_i(x,t) dt^2 + \partial_kg^i(x,t),
\\ \partial_k \hat{g}_1^i&=\frac{-2}{\overline{H}_i(t)^3} \partial_k \bar{H}_i(t) dt^2 + \partial_kg^i(x,t),
\end{align}
where we notice that we already know that, 
\begin{align}
\sum_{k=1}^2 \int_{U_T}|\partial_k\hat{g}^i -\partial_kg^i_1|^2 dV = \int_{U_T}\frac{|\nabla H_i|^2}{H^3_i} dV \rightarrow 0.
\end{align}

\begin{align}
\int_{U_T}|\partial_0\hat{g}^i -\partial_0g^i_1|^2 dV &= \int_{U_T^i} \left |\frac{2}{H_i(x,t)^3} \partial_0 H_i(x,t) -\frac{2}{\overline{H}_i(t)^3} \partial_0 \bar{H}_i(t) \right |^2 dV
\\&= \int_{U_T} \frac{2}{H_i^6\overline{H}_i^6} \left|\bar{H}^3_i \frac{\partial H_i}{\partial t} - H_i^3\frac{\partial \bar{H}_i}{\partial t} \right |^2dV \label{gtog1dtHInt}
\\&\le \frac{2}{H_0^{12}}\int_{U_T}  \bar{H}^3_i \left| \frac{\partial H_i}{\partial t} + \frac{e^{-t/2}}{r_0} \right |^2dV\label{gtog1FirstEq}
\\& +\frac{2}{H_0^{12}} \int_{U_T}\frac{e^{-t}}{r_0^2}\left|\bar{H}^3_i  - H_i^3 \right |^2dV\label{gtog1SecondEq}
\\& +\frac{2}{H_0^{12}}\int_{U_T} H_i^3\left| \frac{e^{-t/2}}{r_0} + \frac{\partial \bar{H}_i}{\partial t} \right |^2 dV\label{gtog1ThirdEq}
\end{align}

Notice that \eqref{gtog1FirstEq} goes to zero since by the evolution equation for $H$ we find
\begin{align}
&\int_{U_T}  \bar{H}^3_i \left| \frac{\partial H_i}{\partial t} + \frac{e^{-t/2}}{r_0} \right |^2dV 
\\&=\int_{U_T}  \bar{H}^3_i \left| \frac{\Delta H_i}{H_i^2} - \frac{|A|_i^2}{H_i} - \frac{Rc^i(\nu,\nu)}{H_i} + \frac{e^{-t/2}}{r_0} \right |^2dV
\\&\le H_1^3 \int_{U_T}\frac{|\Delta H_i|^2}{H_0^2} +\frac{|Rc(\nu,\nu)|^2}{H_0} + \left|\frac{e^{-t/2}}{r_0}- \frac{|A|_i^2}{H_i}\right|^2 dV
\end{align}
and the last equation goes to zero on a subsequence by Corollary \ref{GoToZero} and Lemma \ref{dHdtConvergence}.

 Then we see that \eqref{gtog1SecondEq} goes to zero by Proposition \ref{avgH}, and \eqref{gtog1ThirdEq} goes to zero since Lemma \ref{AvgIntEvol2} implies that $\frac{d \bar{H}_i}{dt}$ is bounded so we can apply the dominated convergence theorem to Corollary \ref{GoToZero2}. The proof is almost exactly the same in all three other cases where the quantity $\frac{e^{-t/2}}{r_0}$ is replaced with the corresponding quantity for the case being considered.
\end{proof}

\begin{Thm}\label{g1tog2} Let $U_{T}^i \subset M_i^3$ be a sequence such that $U_{T}^i\subset \mathcal{M}_{r_0,H_0,I_0}^{T,H_1,A_1}$ and assume
\begin{align}
&m_H(\Sigma_t) \rightarrow 0\text{ as }i \rightarrow \infty,
\\&m_H^{\Hy}(\Sigma_t) \rightarrow 0\text{ as }i \rightarrow \infty,
\\ &m_H(\Sigma_{T}^i)- m_H(\Sigma_{0}^i) \rightarrow 0\text{ and }m_H(\Sigma_t^i)\rightarrow m > 0,
\\ \text{ or }&m_H^{\Hy}(\Sigma_{T}^i)- m_H^{\Hy}(\Sigma_{0}^i) \rightarrow 0\text{ and }m_H^{\Hy}(\Sigma_t^i)\rightarrow m > 0.
\end{align}
 If we define the metrics, 
\begin{align}
g^i_1(x,t)&= \frac{1}{\overline{H}_i(t)^2}dt^2 + g^i(x,t),
\\g^i_2(x,t)&= \frac{1}{\overline{H}_i(t)^2}dt^2 + e^{t-T}g^i(x,T),
\end{align}
 on $U_T^i$ then we have that,
\begin{align}
\int_{U_T}|g^i_1 -g^i_2|_{g_3^i}^2 dV+ \sum_{k=0}^2\int_{U_T}|\partial_k g_1^i -\partial_k g^i_2|^2 dV \rightarrow 0, 
\end{align}
 as $i \rightarrow \infty$ where $dV$ is the volume form on $U_T=\Sigma \times [0,T]$ with respect to $\delta$.
\end{Thm}
\begin{proof}
We already have that
\begin{align}
\int_{U_T^i}&|g^i_1 -g^i_2|^2 dV  \rightarrow 0,\label{Eq-lastg1tog2}
\end{align}
and we can compute that,
\begin{align}
&\int_{U_T}|\partial_0g^i_1 -\partial_0 g^i_2|^2 dV =  \int_{U_T}|\frac{\partial g^i}{\partial t}(x,t) -e^{t-T}g^i(x,T)|^2 dV
\\&=\int_{U_T}|\frac{2A^i}{H_i} -e^{t-T}g^i(x,T)|^2 dV
\\&\le \int_{U_T}  \max_{j=1,2} \left \{\left|\frac{2\lambda_j^i(x,t)g^i(x,t)}{H_i} -e^{t-T}g^i(x,T)\right|^2  \right \}dV
\\&\le \int_{U_T}|g^i(x,T)|^2  \max_{j=1,2} \left \{\left|\frac{2\lambda_j^i(x,t)}{H_i}e^{\int_T^t\frac{2\lambda^i_j(x,s)}{H^i(x,s)}ds} -e^{t-T}\right|^2  \right \}dV,\label{LastEqMetricComparison}
\end{align}
where $\lambda_1^i(x,t), \lambda_2^i(x,t)$ are the smallest and largest eigenvalue of $A^i(x,t)$, respectively. Note by combining Lemma \ref{metricEst} with Lemma \ref{GoToZero} we find that \eqref{LastEqMetricComparison} goes to zero on a subsequence. 

Now we calculate for $k=1,2$ 
\begin{align}
\int_{U_T^i}|\partial_k g^i_1 -\partial_k g^i_2|^2 dV &=  \int_{U_T^i}|\partial_k g^i(x,t) -e^{t-T}\partial_k g^i(x,T)|^2 dV. \label{MetricDerivatives}
\end{align}
Note that by Lemma \ref{C1alphaConvergence} we have that 
\begin{align}
|\partial_kg^i(x,T)| \rightarrow 0,
\end{align} 
and hence by combining Lemma \ref{metricEst}, Lemma \ref{DerivativesOfMetricEstimates}, and Lemma \ref{2ndFundFormDerToZero} we find that 
\begin{align}
\int_0^T\int_{\Sigma}&|\partial_kg^i(x,t)|r_0^2 e^td \mu dt 
\\&\le \int_0^T \int_{\Sigma}\int_0^t \frac{|DA|_i}{H_i} + \frac{|A|_i|DH|_i}{H_i^2} ds r_0^2 e^td \mu dt 
\\& = \int_0^T r_0^2 e^t \int_0^t \int_{\Sigma} \frac{|DA|_i}{H_i} + \frac{|A|_i|DH|_i}{H_i^2}d \mu ds  dt\rightarrow 0.
\end{align} 
Then by applying the result of Lemma \ref{DerivativesOfMetricEstimates} again in a similar way we find that \eqref{MetricDerivatives} goes to zero, as desired. 

We can get rid of the need for a subsequence by assuming to the contrary that for $\epsilon > 0$ there exists a subsequence so that $\int_{U_T^k}|g_1^k -g^k_2|^2 dV \ge \epsilon$, but this subsequence satisfies the hypotheses of Theorem \ref{g1tog2} and hence by what we have just shown we know a subsequence must converge which is a contradiction.
\end{proof}

Now we want to use the fact that we know that the average of the mean curvature is converging to that of a sphere in euclidean space in order to complete the convergence to the warped product $g_3^i$.

\begin{Thm}\label{g2tog3} Let $U_{T}^i \subset M_i^3$ be a sequence such that $U_{T}^i\subset \mathcal{M}_{r_0,H_0,I_0}^{T,H_1,A_1}$ and $m_H(\Sigma_{T}^i) \rightarrow 0$ as $i \rightarrow \infty$. If we define the metrics, \begin{align}
g^i_2(x,t)&= \frac{1}{\bar{H}^i(t)^2}dt^2 + e^{t-T}g^i(x,T),
\\g^i_3(x,t)&= \frac{r_0^2}{4}e^tdt^2 + e^tg^i(x,0),\label{g3Def1}
\end{align}
 on $U_T^i$ then we have that,
\begin{align}
\int_{U_T}|g^i_2 -g^i_3|^2 dV + \sum_{k=0}^2\int_{U_T}|\partial_k g^i_2 -\partial_k g^i_3|^2 dV\rightarrow 0, 
\end{align}
 as $i \rightarrow \infty$ where $dV$ is the volume form on $U_T=\Sigma\times [0,T]$ with respect to $\delta$.

Instead, if $m_H(\Sigma_{T}^i)- m_H(\Sigma_{0}^i) \rightarrow 0$ and $m_H(\Sigma_t^i)\rightarrow m > 0$ and we define, 
\begin{align}
g^i_3(x,t)= \frac{r_0^2}{4}\left (1 - \frac{2}{r_0} m e^{-t/2} \right )^{-1} e^t dt^2 + e^{t-T}g^i(x,T),\label{g3Def2}
\end{align} 
on $U_T^i$ then we have that,
\begin{align}
\int_{U_T}|g^i_2 -g^i_3|^2 dV + \sum_{k=0}^2\int_{U_T}|\partial_k g_2^i -\partial_k g^i_3|^2 dV\rightarrow 0, 
\end{align}
 as $i \rightarrow \infty$ where $dV$ is the volume form on $U_T=\Sigma\times [0,T]$ with respect to $\delta$.

If $m_H^{\Hy}(\Sigma_{T}^i) \rightarrow 0$ as $i \rightarrow \infty$. If we define the metrics, 
\begin{align}
g^i_3(x,t)&= \frac{1}{4}\left(1+ \frac{e^{-t}}{r_0^2} \right)^{-1}dt^2 + e^{t-T}g^i(x,T),\label{g3Def3}
\end{align}
 on $U_T^i$ then we have that,
\begin{align}
\int_{U_T}|g^i_2 -g^i_3|^2 dV + \sum_{k=0}^2\int_{U_T}|\partial_k g^i_2 -\partial_k g^i_3|^2 dV\rightarrow 0 ,
\end{align}
as $i \rightarrow \infty$ where $dV$ is the volume form on $U_T=\Sigma\times [0,T]$ with respect to $\delta$.

Instead, if $m_H^{\Hy}(\Sigma_{T}^i)- m_H^{\Hy}(\Sigma_{0}^i) \rightarrow 0$ and $m_H^{\Hy}(\Sigma_t^i)\rightarrow m > 0$ and we define,
\begin{align}
g^i_3(x,t)= \frac{1}{4}\left (\frac{1}{r_0^2}e^{-t} - \frac{2}{r_0^3} m e^{-3t/2}+1 \right )^{-1} dt^2 + e^{t-T}g^i(x,T),\label{g3Def4}
\end{align} 
on $U_T^i$ then we have that,
\begin{align}
\int_{U_T}|g^i_2 -g^i_3|^2 dV + \sum_{k=0}^2\int_{U_T}|\partial_k g^i_2 -\partial_k g^i_3|^2 dV\rightarrow 0 ,
\end{align}
as $i \rightarrow \infty$ where $dV$ is the volume form on $U_T=\Sigma\times [0,T]$ with respect to $\delta$.

\end{Thm}
\begin{proof}
We have previously shown,
\begin{align}
\int_{U_T}|\hat{g}_2^i -g^i_3|^2 dV   \rightarrow 0, \label{Eq-last1}
\end{align}
and we note that the only coordinate derivatives we need to consider in this case are with respect to $t$,
\begin{align}
\partial_0 g_2^i&=\frac{-2}{\bar{H}_i(t)^3} \frac{\partial \bar{H}_i}{\partial t}(t) dt^2 + e^{t-T}g^i(x,T),
\\ \partial_0 g_3^i &=\frac{r_0^2}{4}e^t dt^2 + e^{t-T}g^i(x,T),
\end{align}
and so we compute
\begin{align}
\int_{U_T}|\partial_0\hat{g}_2^i -\partial_0g^i_3|^2 dV   &=\int_{U_T}\left |\frac{r_0^2}{4}e^t +\frac{2}{\bar{H}_i(t)^3} \frac{\partial \bar{H_i}}{\partial t}(x,t) \right|^2\rightarrow 0,\label{g2tog3dtavgHEst}
\end{align}
which follows from applying Corollary \ref{GoToZero2} and Lemma \ref{naiveEstimate} which give pointwise almost everywhere convergence of the integrand in  \eqref{g2tog3dtavgHEst} on a subsequence. Notice that Lemma \ref{AvgHEst} also implies that $\frac{d\bar{H}^2}{dt}$ is bounded and hence we can use the dominated convergence theorem to conclude that the integral in \eqref{g2tog3dtavgHEst} is $\rightarrow 0$.  The proof is almost exactly the same in all three other cases where the quantity $\frac{r_0^2}{4}e^t$ is replaced with the corresponding quantity for the case being considered.

We can get rid of the need for a subsequence by assuming to the contrary that for $\epsilon > 0$ there exists a subsequence so that $\int_{U_T}|g_2^k -g^k_3|^2 dV \ge \epsilon$, but this subsequence satisfies the hypotheses of Theorem \ref{g2tog3} and hence by what we have just shown we know a subsequence must converge which is a contradiction.
\end{proof}

We now finish by completing the $W^{1,2}$ convergence to the appropriate prototype spaces.

\begin{Thm}\label{g3todelta} 
Let $U_{T}^i \subset M_i^3$ be a sequence such that $U_{T}^i\subset \mathcal{M}_{r_0,H_0,I_0}^{T,H_1,A_1}$ and $m_H(\Sigma_{T}^i) \rightarrow 0$ as $i \rightarrow \infty$. Considering the metric \eqref{g3Def1} on $U_T^i$ then, 
\begin{align}
\int_{U_T}|g^i_3 -\delta|^2 dV + \sum_{k=0}^2\int_{U_T}|\partial_k g^i_3 -\partial_k \delta|^2 dV\rightarrow 0,
\end{align}
 as $i \rightarrow \infty$ where $dV$ is the volume form on $U_T=\Sigma\times [0,T]$ with respect to $\delta$.

Instead, if $m_H(\Sigma_{T}^i)- m_H(\Sigma_{0}^i) \rightarrow 0$ and $m_H(\Sigma_t^i)\rightarrow m > 0$ and we consider \eqref{g3Def2}
on $U_T^i$ then,
\begin{align}
\int_{U_T}|g^i_3 -g_S|^2 dV + \sum_{k=0}^2\int_{U_T}|\partial_k g_3^i -\partial_k g_S|^2 dV\rightarrow 0 ,
\end{align}
 as $i \rightarrow \infty$ where $dV$ is the volume form on $U_T=\Sigma\times [0,T]$ with respect to $\delta$.

If $m_H^{\Hy}(\Sigma_{T}^i) \rightarrow 0$ as $i \rightarrow \infty$ and we consider \eqref{g3Def3}
 on $U_T^i$ then,
\begin{align}
\int_{U_T}|g^i_3 -g_{\Hy}|^2 dV + \sum_{k=0}^2\int_{U_T}|\partial_k g^i_3 -\partial_k g_{\Hy}|^2 dV\rightarrow 0,
\end{align}
as $i \rightarrow \infty$ where $dV$ is the volume form on $U_T=\Sigma\times [0,T]$ with respect to $\delta$.

Instead, if $m_H^{\Hy}(\Sigma_{T}^i)- m_H^{\Hy}(\Sigma_{0}^i) \rightarrow 0$ and $m_H^{\Hy}(\Sigma_t^i)\rightarrow m > 0$ and we consider \eqref{g3Def4} on $U_T^i$ then,
\begin{align}
\int_{U_T}|g^i_3 -g_{ADSS}|^2 dV + \sum_{k=0}^2\int_{U_T}|\partial_k g^i_3 -\partial_k g_{ADSS}|^2 dV\rightarrow 0,
\end{align}
as $i \rightarrow \infty$ where $dV$ is the volume form on $U_T=\Sigma\times [0,T]$ with respect to $\delta$.
\end{Thm}
\begin{proof}
This follows by Lemma \ref{C1alphaConvergence}.
\end{proof}

\textbf{Proof of Main Theorems:}
\begin{proof}
The proof of the main theorems now follows by combining Theorem \ref{gtog1}, Theorem \ref{g1tog2}, Theorem \ref{g2tog3} and Theorem \ref{g3todelta} via the triangle inequality. Note that $W^{1,2}$ convergence implies $L^6$ convergence in dimension $3$ which is how we can conclude that the volume of regions are converging.
\end{proof}

\end{document}